\documentclass{amsart}
\usepackage{amssymb}
\usepackage{amsfonts}

\newtheorem{theorem}{Theorem}
\theoremstyle{plain}

\newtheorem{corollary}{Corollary}

\newtheorem{lemma}{Lemma}

\newtheorem{proposition}{Proposition}
\newtheorem{remark}{Remark}

\numberwithin{equation}{section}
\input{tcilatex}

\begin{document}
\title[Inequalities For $\log -$ Convex Functions]{Inequalities for $\log -$%
convex functions via three times differentiability}
\author{Merve Avc\i\ Ard\i \c{c}$^{\clubsuit ,\bigstar }$}
\address{$^{\clubsuit }$ADIYAMAN UNIVERSITY, FACULTY OF SCIENCE AND ARTS,
DEPARTMENT OF MATHEMATICS, 02040 ADIYAMAN}
\email{merveavci@ymail.com}
\thanks{$^{\bigstar }$Corresponding Author}
\author{M. Emin \"{O}zdemir$^{\diamondsuit }$}
\email{emos@atauni.edu.tr}
\address{$^{\diamondsuit }$ATAT\"{U}RK UNIVERSITY, K.K EDUCATION FACULTY,
DEPARTMENT OF MATHEMATICS, 25240, ERZURUM}

\begin{abstract}
In this paper, we obtain some new integral inequalities like
Hermite-Hadamard type for third derivatives absolute value are $\log -$%
convex. \ We give some applications to quadrature formula for midpoint error
estimate.
\end{abstract}

\maketitle

\section{INTRODUCTION}

We shall recall the definitions of convex functions and $\log -$convex
functions:

Let $I$ be an interval in $%
\mathbb{R}
.$ Then $f:I\rightarrow 
\mathbb{R}
$ is said to be convex if for all on $x,y\in I$ \ and all $\alpha \in
\lbrack 0,1],$%
\begin{equation}
f(\alpha x+(1-\alpha )y)\leq \alpha f(x)+(1-\alpha )f(y)  \label{1.1}
\end{equation}

holds. If (\ref{1.1}) is strict for all $x\neq y$ and $\alpha \in (0,1),$
then $f$ is said to be strictly convex. If the inequality in (\ref{1.1}) is
reversed, then $f$ is said to be concave. If it is strict for all $x\neq y$
and $\alpha \in (0,1),$ then $f$ is said to be strictly concave.

A function is called $\log -$convex or multiplicatively convex on a real
interval $I=[a,b],$ if $\log f$ is convex , or, equivalently if for all $%
x,y\in I$ \ and all $\alpha \in \lbrack 0,1],$%
\begin{equation}
f(\alpha x+(1-\alpha )y)\leq f(x)^{\alpha }+f(y)^{(1-\alpha )}.  \label{1.2}
\end{equation}%
It is said to be $\log -$concave if the inequality in (\ref{1.2}) is
reversed.

For some results for $\log -$convex functions see \cite{1}-\cite{4}.

The following inequality is called Hermite-Hadamard inequality for convex
functions:

Let $f:I\subseteq 
\mathbb{R}
\rightarrow 
\mathbb{R}
$ be a convex function on the interval $I$ of real numbers and $a,b\in I$
with $a<b.$ Then double inequality

\begin{equation*}
f\left( \frac{a+b}{2}\right) \leq \frac{1}{b-a}\int_{a}^{b}f(x)dx\leq \frac{%
f(a)+f(b)}{2}
\end{equation*}%
holds.

The main purpose of this paper is  to obtain some new integral inequalities
like Hermite-Hadamard type for third derivatives absolute value are $\log -$%
convex.

In order to prove our main results for $\log -$convex functios we need the
following Lemma from \cite{5}:

\begin{lemma}
\label{lem 1.1} Let $f:I\subset 
\mathbb{R}
\rightarrow 
\mathbb{R}
$ be a three times differentiable mapping on $I^{\circ }$ and $a,b\in
I^{\circ }$ with $a<b.$ If $f^{\left( 3\right) }\in L_{1}([a,b]),$ then%
\begin{eqnarray*}
&&\frac{1}{b-a}\int_{a}^{b}f(x)dx-f\left( \frac{a+b}{2}\right) -\frac{\left(
b-a\right) ^{2}}{24}f^{\prime \prime }\left( \frac{a+b}{2}\right) \\
&=&\frac{\left( b-a\right) ^{3}}{96}\left[ \int_{0}^{1}t^{3}f^{\left(
3\right) }\left( \frac{t}{2}a+\frac{2-t}{2}b\right)
dt-\int_{0}^{1}t^{3}f^{\left( 3\right) }\left( \frac{2-t}{2}a+\frac{t}{2}%
b\right) dt\right] .
\end{eqnarray*}
\end{lemma}

\section{INEQUALITIES FOR $\log -$CONVEX FUNCTIONS}

We shall start the following result:

\begin{theorem}
\label{teo 2.1} Let $f:I\rightarrow \lbrack 0,\infty ),$ be a three times
differentiable mapping on $I^{\circ \text{ }}$ such that $f^{\prime \prime
\prime }\in L[a,b]$ where $a,b$ $\in I^{\circ }$ with $a<b.$ If $\left\vert
f^{\prime \prime \prime }\right\vert $ is $\log -$convex on $[a,b]$ , then
the following inequality holds:%
\begin{eqnarray*}
&&\left\vert \frac{1}{b-a}\int_{a}^{b}f(x)dx-f\left( \frac{a+b}{2}\right) -%
\frac{\left( b-a\right) ^{2}}{24}f^{\prime \prime }\left( \frac{a+b}{2}%
\right) \right\vert \\
&\leq &\frac{\left( b-a\right) ^{3}}{96}\left\{ \left\vert f^{\prime \prime
\prime }(b)\right\vert \mu _{K}+\left\vert f^{\prime \prime \prime
}(a)\right\vert \mu _{M}\right\}
\end{eqnarray*}%
where 
\begin{equation*}
\mu _{K}=\frac{2K^{\frac{1}{2}}\left( \ln K-6\right) }{\left( \ln K\right)
^{2}}+\frac{48K^{\frac{1}{2}}\left( \ln K-2\right) }{\left( \ln K\right) ^{4}%
}+\frac{96}{\left( \ln K\right) ^{4}},
\end{equation*}%
\begin{equation*}
\mu _{M}=\frac{2M^{\frac{1}{2}}\left( \ln M-6\right) }{\left( \ln M\right)
^{2}}+\frac{48M^{\frac{1}{2}}\left( \ln M-2\right) }{\left( \ln M\right) ^{4}%
}+\frac{96}{\left( \ln M\right) ^{4}}
\end{equation*}%
and%
\begin{equation*}
K=\frac{\left\vert f^{\prime \prime \prime }(a)\right\vert }{\left\vert
f^{\prime \prime \prime }(b)\right\vert },\text{ }M=\frac{\left\vert
f^{\prime \prime \prime }(b)\right\vert }{\left\vert f^{\prime \prime \prime
}(a)\right\vert }.
\end{equation*}
\end{theorem}

\begin{proof}
From Lemma \ref{lem 1.1}, property of the modulus and $\log -$convexity of $%
\left\vert f^{\prime \prime \prime }\right\vert $ we have%
\begin{eqnarray*}
&&\left\vert \frac{1}{b-a}\int_{a}^{b}f(x)dx-f\left( \frac{a+b}{2}\right) -%
\frac{\left( b-a\right) ^{2}}{24}f^{\prime \prime }\left( \frac{a+b}{2}%
\right) \right\vert \\
&\leq &\frac{\left( b-a\right) ^{3}}{96}\left\{ \int_{0}^{1}t^{3}\left\vert
f^{\prime \prime \prime }\left( \frac{t}{2}a+\frac{2-t}{2}b\right)
\right\vert dt+\int_{0}^{1}t^{3}\left\vert f^{\prime \prime \prime }\left( 
\frac{t}{2}b+\frac{2-t}{2}a\right) \right\vert dt\right\} \\
&\leq &\frac{\left( b-a\right) ^{3}}{96}\left\{ \int_{0}^{1}t^{3}\left\vert
f^{\prime \prime \prime }(a)\right\vert ^{\frac{t}{2}}\left\vert f^{\prime
\prime \prime }(b)\right\vert ^{1-\frac{t}{2}}dt+\int_{0}^{1}t^{3}\left\vert
f^{\prime \prime \prime }(b)\right\vert ^{\frac{t}{2}}\left\vert f^{\prime
\prime \prime }(a)\right\vert ^{1-\frac{t}{2}}dt\right\} \\
&=&\frac{\left( b-a\right) ^{3}}{96}\left\{ \left\vert f^{\prime \prime
\prime }(b)\right\vert \int_{0}^{1}t^{3}\left[ \frac{\left\vert f^{\prime
\prime \prime }(a)\right\vert }{\left\vert f^{\prime \prime \prime
}(b)\right\vert }\right] ^{\frac{t}{2}}dt+\left\vert f^{\prime \prime \prime
}(a)\right\vert \int_{0}^{1}t^{3}\left[ \frac{\left\vert f^{\prime \prime
\prime }(b)\right\vert }{\left\vert f^{\prime \prime \prime }(a)\right\vert }%
\right] ^{\frac{t}{2}}dt\right\} .
\end{eqnarray*}%
The proof is completed by making use of the neccessary computation.
\end{proof}

\begin{corollary}
\label{co 1.1.} Let $\mu _{K},\mu _{M},K$ and $\ M$ be defined as in Theorem %
\ref{teo 2.1}. If we choose $f^{\prime \prime }\left( \frac{a+b}{2}\right)
=0 $ in Theorem \ref{teo 2.1}, we obtain the following inequality%
\begin{eqnarray*}
&&\left\vert \frac{1}{b-a}\int_{a}^{b}f(x)dx-f\left( \frac{a+b}{2}\right)
\right\vert \\
&\leq &\frac{\left( b-a\right) ^{3}}{96}\left\{ \left\vert f^{\prime \prime
\prime }(b)\right\vert \mu _{K}+\left\vert f^{\prime \prime \prime
}(a)\right\vert \mu _{M}\right\} .
\end{eqnarray*}
\end{corollary}

\begin{theorem}
\label{teo 2.2} Let $f:I\rightarrow \lbrack 0,\infty ),$ be a three times
differentiable mapping on $I^{\circ \text{ }}$ such that $f^{\prime \prime
\prime }\in L[a,b]$ where $a,b$ $\in I^{\circ }$ with $a<b.$ If $\left\vert
f^{\prime \prime \prime }\right\vert ^{q}$ is $\log -$convex on $[a,b]$ ,
then the following inequality holds for some fixed $q>1$%
\begin{eqnarray*}
&&\left\vert \frac{1}{b-a}\int_{a}^{b}f(x)dx-f\left( \frac{a+b}{2}\right) -%
\frac{\left( b-a\right) ^{2}}{24}f^{\prime \prime }\left( \frac{a+b}{2}%
\right) \right\vert \\
&\leq &\frac{\left( b-a\right) ^{3}}{96}\left( \frac{1}{3p+1}\right) ^{\frac{%
1}{p}}\left\{ \left\vert f^{\prime \prime \prime }(b)\right\vert \left( 
\frac{2}{q\ln K}\left[ K^{\frac{q}{2}}-1\right] \right) ^{\frac{1}{q}%
}+\left\vert f^{\prime \prime \prime }(a)\right\vert \left( \frac{2}{q\ln M}%
\left[ M^{\frac{q}{2}}-1\right] \right) ^{\frac{1}{q}}\right\}
\end{eqnarray*}%
where $K=\frac{\left\vert f^{\prime \prime \prime }(a)\right\vert }{%
\left\vert f^{\prime \prime \prime }(b)\right\vert },$ $M=\frac{\left\vert
f^{\prime \prime \prime }(b)\right\vert }{\left\vert f^{\prime \prime \prime
}(a)\right\vert }$ and $\frac{1}{p}+\frac{1}{q}=1.$
\end{theorem}

\begin{proof}
From Lemma \ref{lem 1.1} and using the H\"{o}lder's integral inequality, we
obtain%
\begin{eqnarray*}
&&\left\vert \frac{1}{b-a}\int_{a}^{b}f(x)dx-f\left( \frac{a+b}{2}\right) -%
\frac{\left( b-a\right) ^{2}}{24}f^{\prime \prime }\left( \frac{a+b}{2}%
\right) \right\vert \\
&\leq &\frac{\left( b-a\right) ^{3}}{96}\left\{ \left(
\int_{0}^{1}t^{3p}dt\right) ^{\frac{1}{p}}\left( \int_{0}^{1}\left\vert
f^{\prime \prime \prime }\left( \frac{t}{2}a+\frac{2-t}{2}b\right)
\right\vert ^{q}dt\right) ^{\frac{1}{q}}\right. \\
&&\left. +\left( \int_{0}^{1}t^{3p}dt\right) ^{\frac{1}{p}}\left(
\int_{0}^{1}\left\vert f^{\prime \prime \prime }\left( \frac{t}{2}b+\frac{2-t%
}{2}a\right) \right\vert ^{q}dt\right) ^{\frac{1}{q}}\right\} .
\end{eqnarray*}%
If we use the $\log $-convexity of $\left\vert f^{\prime \prime \prime
}\right\vert ^{q}$ above, we can write%
\begin{eqnarray*}
&&\left\vert \frac{1}{b-a}\int_{a}^{b}f(x)dx-f\left( \frac{a+b}{2}\right) -%
\frac{\left( b-a\right) ^{2}}{24}f^{\prime \prime }\left( \frac{a+b}{2}%
\right) \right\vert \\
&\leq &\frac{\left( b-a\right) ^{3}}{96}\left\{ \left(
\int_{0}^{1}t^{3p}dt\right) ^{\frac{1}{p}}\left( \int_{0}^{1}\left\vert
f^{\prime \prime \prime }(a)\right\vert ^{\frac{qt}{2}}\left\vert f^{\prime
\prime \prime }(b)\right\vert ^{q-\frac{qt}{2}}dt\right) ^{\frac{1}{q}%
}\right. \\
&&+\left. \left( \int_{0}^{1}t^{3p}dt\right) ^{\frac{1}{p}}\left(
\int_{0}^{1}\left\vert f^{\prime \prime \prime }(b)\right\vert ^{\frac{qt}{2}%
}\left\vert f^{\prime \prime \prime }(a)\right\vert ^{q-\frac{qt}{2}%
}dt\right) ^{\frac{1}{q}}\right\} \\
&=&\frac{\left( b-a\right) ^{3}}{96}\left( \frac{1}{3p+1}\right) ^{\frac{1}{p%
}} \\
&&\times \left\{ \left\vert f^{\prime \prime \prime }(b)\right\vert \left( 
\frac{2}{q\ln K}\left[ K^{\frac{q}{2}}-1\right] \right) ^{\frac{1}{q}%
}+\left\vert f^{\prime \prime \prime }(a)\right\vert \left( \frac{2}{q\ln M}%
\left[ M^{\frac{q}{2}}-1\right] \right) ^{\frac{1}{q}}\right\}
\end{eqnarray*}%
where 
\begin{equation*}
\int_{0}^{1}t^{3p}dt=\frac{1}{3p+1},
\end{equation*}%
\begin{equation*}
\int_{0}^{1}\left\vert f^{\prime \prime \prime }(a)\right\vert ^{\frac{qt}{2}%
}\left\vert f^{\prime \prime \prime }(b)\right\vert ^{q-\frac{qt}{2}%
}dt=\left\vert f^{\prime \prime \prime }(b)\right\vert ^{q}\left( \frac{2}{%
q\ln K}\left[ K^{\frac{q}{2}}-1\right] \right)
\end{equation*}%
and%
\begin{equation*}
\int_{0}^{1}\left\vert f^{\prime \prime \prime }(b)\right\vert ^{\frac{qt}{2}%
}\left\vert f^{\prime \prime \prime }(a)\right\vert ^{q-\frac{qt}{2}%
}dt=\left\vert f^{\prime \prime \prime }(a)\right\vert ^{q}\left( \frac{2}{%
q\ln M}\left[ M^{\frac{q}{2}}-1\right] \right) .
\end{equation*}%
The proof is completed.
\end{proof}

\begin{corollary}
\label{co 1.2.} Let $K$ and $M$ be defined as in Theorem \ref{teo 2.2}. If
we choose $f^{\prime \prime }\left( \frac{a+b}{2}\right) =0$ in Theorem \ref%
{teo 2.2}, we obtain the following inequality%
\begin{eqnarray*}
&&\left\vert \frac{1}{b-a}\int_{a}^{b}f(x)dx-f\left( \frac{a+b}{2}\right)
\right\vert \\
&\leq &\frac{\left( b-a\right) ^{3}}{96}\left( \frac{1}{3p+1}\right) ^{\frac{%
1}{p}}\left\{ \left\vert f^{\prime \prime \prime }(b)\right\vert \left( 
\frac{2}{q\ln K}\left[ K^{\frac{q}{2}}-1\right] \right) ^{\frac{1}{q}%
}+\left\vert f^{\prime \prime \prime }(a)\right\vert \left( \frac{2}{q\ln M}%
\left[ M^{\frac{q}{2}}-1\right] \right) ^{\frac{1}{q}}\right\}
\end{eqnarray*}%
where $q>1,$ $\frac{1}{p}+\frac{1}{q}=1.$
\end{corollary}

\begin{theorem}
\label{teo 2.3} Let $f:I\rightarrow \lbrack 0,\infty ),$ be a three times
differentiable mapping on $I^{\circ \text{ }}$ such that $f^{\prime \prime
\prime }\in L[a,b]$ where $a,b$ $\in I^{\circ }$ with $a<b.$ If $\left\vert
f^{\prime \prime \prime }\right\vert ^{q}$ is $\log -$convex on $[a,b]$ ,
then the following inequality holds for some fixed $q\geq 1,$ then the
following inequality holds:%
\begin{eqnarray*}
&&\left\vert \frac{1}{b-a}\int_{a}^{b}f(x)dx-f\left( \frac{a+b}{2}\right) -%
\frac{\left( b-a\right) ^{2}}{24}f^{\prime \prime }\left( \frac{a+b}{2}%
\right) \right\vert \\
&\leq &\frac{\left( b-a\right) ^{3}}{96}\left( \frac{1}{4}\right) ^{1-\frac{1%
}{q}}\left\{ \left\vert f^{\prime \prime \prime }(b)\right\vert \left( \mu
_{K,q}\right) ^{\frac{1}{q}}+\left\vert f^{\prime \prime \prime
}(a)\right\vert \left( \mu _{M,q}\right) ^{\frac{1}{q}}\right\}
\end{eqnarray*}%
\begin{equation*}
\mu _{K,q}=\frac{2K^{\frac{q}{2}}\left( q\ln K-6\right) }{\left( q\ln
K\right) ^{2}}+\frac{48K^{\frac{q}{2}}\left( q\ln K-2\right) }{\left( q\ln
K\right) ^{4}}+\frac{96}{\left( q\ln K\right) ^{4}},
\end{equation*}%
\begin{equation*}
\mu _{M,q}=\frac{2M^{\frac{q}{2}}\left( q\ln M-6\right) }{\left( q\ln
M\right) ^{2}}+\frac{48M^{\frac{q}{2}}\left( q\ln M-2\right) }{\left( q\ln
M\right) ^{4}}+\frac{96}{\left( q\ln M\right) ^{4}}
\end{equation*}%
and%
\begin{equation*}
K=\frac{\left\vert f^{\prime \prime \prime }(a)\right\vert }{\left\vert
f^{\prime \prime \prime }(b)\right\vert },\text{ }M=\frac{\left\vert
f^{\prime \prime \prime }(b)\right\vert }{\left\vert f^{\prime \prime \prime
}(a)\right\vert }.
\end{equation*}
\end{theorem}

\begin{proof}
From Lemma \ref{lem 1.1}, using the well-known power-mean integral
inequality and $\log -$convexity of $\left\vert f^{\prime \prime \prime
}\right\vert ^{q}$ we have%
\begin{eqnarray*}
&&\left\vert \frac{1}{b-a}\int_{a}^{b}f(x)dx-f\left( \frac{a+b}{2}\right) -%
\frac{\left( b-a\right) ^{2}}{24}f^{\prime \prime }\left( \frac{a+b}{2}%
\right) \right\vert \\
&\leq &\frac{\left( b-a\right) ^{3}}{96}\left\{ \left(
\int_{0}^{1}t^{3}dt\right) ^{1-\frac{1}{q}}\left(
\int_{0}^{1}t^{3}\left\vert f^{\prime \prime \prime }\left( \frac{t}{2}a+%
\frac{2-t}{2}b\right) \right\vert ^{q}dt\right) ^{\frac{1}{q}}\right. \\
&&\left. +\left( \int_{0}^{1}t^{3}dt\right) ^{1-\frac{1}{q}}\left(
\int_{0}^{1}t^{3}\left\vert f^{\prime \prime \prime }\left( \frac{t}{2}b+%
\frac{2-t}{2}a\right) \right\vert ^{q}dt\right) ^{\frac{1}{q}}\right\} \\
&\leq &\frac{\left( b-a\right) ^{3}}{96}\left\{ \left(
\int_{0}^{1}t^{3}dt\right) ^{1-\frac{1}{q}}\left(
\int_{0}^{1}t^{3}\left\vert f^{\prime \prime \prime }(a)\right\vert ^{\frac{%
qt}{2}}\left\vert f^{\prime \prime \prime }(b)\right\vert ^{q-\frac{qt}{2}%
}dt\right) ^{\frac{1}{q}}\right. \\
&&\left. +\left( \int_{0}^{1}t^{3}dt\right) ^{1-\frac{1}{q}}\left(
\int_{0}^{1}t^{3}\left\vert f^{\prime \prime \prime }(b)\right\vert ^{\frac{%
qt}{2}}\left\vert f^{\prime \prime \prime }(a)\right\vert ^{q-\frac{qt}{2}%
}dt\right) ^{\frac{1}{q}}\right\} .
\end{eqnarray*}%
The proof is completed by making use of the neccessary computation.
\end{proof}

\begin{corollary}
\label{co 1.3.} Let $\mu _{K,q},\mu _{M,q},K$ and $\ M$ be defined as in
Theorem \ref{teo 2.3}. If we choose $f^{\prime \prime }\left( \frac{a+b}{2}%
\right) =0$ in Theorem \ref{teo 2.3}, we obtain the following inequality%
\begin{eqnarray*}
&&\left\vert \frac{1}{b-a}\int_{a}^{b}f(x)dx-f\left( \frac{a+b}{2}\right)
\right\vert \\
&\leq &\frac{\left( b-a\right) ^{3}}{96}\left( \frac{1}{4}\right) ^{1-\frac{1%
}{q}}\left\{ \left\vert f^{\prime \prime \prime }(b)\right\vert \left( \mu
_{K,q}\right) ^{\frac{1}{q}}+\left\vert f^{\prime \prime \prime
}(a)\right\vert \left( \mu _{M,q}\right) ^{\frac{1}{q}}\right\} .
\end{eqnarray*}
\end{corollary}

\begin{corollary}
\label{co 1.4.} From Corollaries \ref{co 1.1.}-\ref{co 1.3.}, we have%
\begin{equation*}
\left\vert \frac{1}{b-a}\int_{a}^{b}f(x)dx-f\left( \frac{a+b}{2}\right)
\right\vert \leq \min \left\{ \chi _{1},\chi _{2},\chi _{3}\right\}
\end{equation*}%
where%
\begin{eqnarray*}
\chi _{1} &=&\frac{\left( b-a\right) ^{3}}{96}\left\{ \left\vert f^{\prime
\prime \prime }(b)\right\vert \frac{2K^{\frac{1}{2}}\left( \ln K-6\right) }{%
\left( \ln K\right) ^{2}}+\frac{48K^{\frac{1}{2}}\left( \ln K-2\right) }{%
\left( \ln K\right) ^{4}}+\frac{96}{\left( \ln K\right) ^{4}}\right. \\
&&\left. +\left\vert f^{\prime \prime \prime }(a)\right\vert \frac{2M^{\frac{%
1}{2}}\left( \ln M-6\right) }{\left( \ln M\right) ^{2}}+\frac{48M^{\frac{1}{2%
}}\left( \ln M-2\right) }{\left( \ln M\right) ^{4}}+\frac{96}{\left( \ln
M\right) ^{4}}\right\} ,
\end{eqnarray*}%
\begin{eqnarray*}
\chi _{2} &=&\frac{\left( b-a\right) ^{3}}{96}\left( \frac{1}{3p+1}\right) ^{%
\frac{1}{p}} \\
&&\times \left\{ \left\vert f^{\prime \prime \prime }(b)\right\vert \left( 
\frac{2}{q\ln K}\left[ K^{\frac{q}{2}}-1\right] \right) ^{\frac{1}{q}%
}+\left\vert f^{\prime \prime \prime }(a)\right\vert \left( \frac{2}{q\ln M}%
\left[ M^{\frac{q}{2}}-1\right] \right) ^{\frac{1}{q}}\right\} ,
\end{eqnarray*}%
\begin{eqnarray*}
\chi _{3} &=&\frac{\left( b-a\right) ^{3}}{96}\left( \frac{1}{4}\right) ^{1-%
\frac{1}{q}}\left\{ \left\vert f^{\prime \prime \prime }(b)\right\vert
\left( \frac{2K^{\frac{q}{2}}\left( q\ln K-6\right) }{\left( q\ln K\right)
^{2}}+\frac{48K^{\frac{q}{2}}\left( q\ln K-2\right) }{\left( q\ln K\right)
^{4}}+\frac{96}{\left( q\ln K\right) ^{4}}\right) ^{\frac{1}{q}}\right. \\
&&\left. +\left\vert f^{\prime \prime \prime }(a)\right\vert \left( \frac{%
2M^{\frac{1}{2}}\left( \ln M-6\right) }{\left( \ln M\right) ^{2}}+\frac{48M^{%
\frac{1}{2}}\left( \ln M-2\right) }{\left( \ln M\right) ^{4}}+\frac{96}{%
\left( \ln M\right) ^{4}}\right) ^{\frac{1}{q}}\right\}
\end{eqnarray*}%
and $K=\frac{\left\vert f^{\prime \prime \prime }(a)\right\vert }{\left\vert
f^{\prime \prime \prime }(b)\right\vert },$ $M=\frac{\left\vert f^{\prime
\prime \prime }(b)\right\vert }{\left\vert f^{\prime \prime \prime
}(a)\right\vert }.$
\end{corollary}

\begin{remark}
\label{rem 1.1} In Theorem \ref{teo 2.3} and Corollary \ref{co 1.3.}, if we
choose $q=1,$ we obtain Theorem \ref{teo 2.1} and Corollary \ref{co 1.1.}
respectively.
\end{remark}

\section{APPLICATIONS TO MIDPOINT FORMULA}

We give some error estimates to midpoint formula by using the results of
Section 2.

Let $d$ be a division $a=x_{0}<x_{1}<...<x_{n-1}<x_{n}=b$ of the interval $%
[a,b]$ and consider the formula%
\begin{equation*}
\int_{a}^{b}f(x)dx=M(f,d)+E(f,d)
\end{equation*}%
where $M(f,d)=\sum_{i=0}^{n-1}f\left( \frac{x_{i}+x_{i+1}}{2}\right) \left(
x_{i+1}-x_{i}\right) $ for the midpoint version and $E(f,d)$ denotes the
associated approximation error.

\begin{proposition}
\label{prop 3.1} Let $f:I\rightarrow \lbrack 0,\infty )$ be a three times
differentiable mapping on $I^{\circ }$ with $a,b\in I^{\circ }$ such that $%
a<b.$ If $\left\vert f^{\prime \prime \prime }\right\vert $ is  $\log -$%
convex function with $f^{\prime \prime \prime }\in L_{1}([a,b]),$ then for
every division $d$ of $[a,b],$ the midpoint error estimate satisfies%
\begin{equation*}
\left\vert E(f,d)\leq \right\vert \sum_{i=0}^{n-1}\frac{\left(
x_{i+1}-x_{i}\right) ^{4}}{96}\left\{ \left\vert f^{\prime \prime \prime
}(x_{i+1})\right\vert \mu _{1}+\left\vert f^{\prime \prime \prime
}(x_{i})\right\vert \mu _{2}\right\} 
\end{equation*}%
where 
\begin{equation*}
\mu _{1}=\frac{2K_{1}^{\frac{1}{2}}\left( \ln K_{1}-6\right) }{\left( \ln
K_{1}\right) ^{2}}+\frac{48K_{1}^{\frac{1}{2}}\left( \ln K_{1}-2\right) }{%
\left( \ln K_{1}\right) ^{4}}+\frac{96}{\left( \ln K_{1}\right) ^{4}},
\end{equation*}%
\begin{equation*}
\mu _{2}=\frac{2M_{1}^{\frac{1}{2}}\left( \ln M_{1}-6\right) }{\left( \ln
M_{1}\right) ^{2}}+\frac{48M_{1}^{\frac{1}{2}}\left( \ln M_{1}-2\right) }{%
\left( \ln M_{1}\right) ^{4}}+\frac{96}{\left( \ln M_{1}\right) ^{4}}
\end{equation*}%
and%
\begin{equation*}
K_{1}=\frac{\left\vert f^{\prime \prime \prime }(x_{i})\right\vert }{%
\left\vert f^{\prime \prime \prime }(x_{i+1})\right\vert },\text{ }M_{1}=%
\frac{\left\vert f^{\prime \prime \prime }(x_{i+1})\right\vert }{\left\vert
f^{\prime \prime \prime }(x_{i})\right\vert }.
\end{equation*}
\end{proposition}

\begin{proof}
By appliying Corollary \ref{co 1.1.} on the subintervals $[x_{i},x_{i+1}],$ $%
(i=0,1,...,n-1)$ of the division $d$, we have%
\begin{eqnarray*}
&&\left\vert \frac{1}{x_{i+1}-x_{i}}\int_{x_{i}}^{x_{i+1}}f(x)dx-f\left( 
\frac{x_{i}+x_{i+1}}{2}\right) \right\vert  \\
&\leq &\frac{\left( x_{i+1}-x_{i}\right) ^{3}}{96}\left\{ \left\vert
f^{\prime \prime \prime }(x_{i+1})\right\vert \mu _{1}+\left\vert f^{\prime
\prime \prime }(x_{i})\right\vert \mu _{2}\right\} .
\end{eqnarray*}%
By summing over $i$ from $0$ to $n-1,$ we can write%
\begin{equation*}
\left\vert \int_{a}^{b}f(x)dx-M(f,d)\right\vert \leq \sum_{i=0}^{n-1}\frac{%
\left( x_{i+1}-x_{i}\right) ^{4}}{96}\left\{ \left\vert f^{\prime \prime
\prime }(x_{i+1})\right\vert \mu _{1}+\left\vert f^{\prime \prime \prime
}(x_{i})\right\vert \mu _{2}\right\} 
\end{equation*}%
which completes the proof.
\end{proof}

\begin{proposition}
\label{prop 3.2} Let $f:I\rightarrow \lbrack 0,\infty )$ be a three times
differentiable mapping on $I^{\circ }$ with $a,b\in I^{\circ }$ such that $%
a<b.$ If $\left\vert f^{\prime \prime \prime }\right\vert ^{q}$ is $\log -$%
convex function with $f^{\prime \prime \prime }\in L_{1}([a,b])$ for some
fixed $q>1,$ then for every division $d$ of $[a,b],$ the midpoint error
estimate satisfies%
\begin{eqnarray*}
&&\left\vert E(f,d)\leq \right\vert \left( \frac{1}{3p+1}\right) ^{\frac{1}{p%
}}\frac{1}{96}\sum_{i=0}^{n-1}\left( x_{i+1}-x_{i}\right) ^{4}\left\{
\left\vert f^{\prime \prime \prime }(x_{i+1})\right\vert \left( \frac{2}{%
q\ln K_{1}}\left[ K_{1}^{\frac{q}{2}}-1\right] \right) ^{\frac{1}{q}}\right. 
\\
&&\left. +\left\vert f^{\prime \prime \prime }(x_{i})\right\vert \left( 
\frac{2}{q\ln M_{1}}\left[ M_{1}^{\frac{q}{2}}-1\right] \right) ^{\frac{1}{q}%
}\right\} 
\end{eqnarray*}%
where $\frac{1}{p}+\frac{1}{q}=1$ and $K_{1},M_{1}$ are as defined in
Proposition \ref{prop 3.1}.
\end{proposition}

\begin{proof}
The proof can be maintained by using Corollary \ref{co 1.2.} like
Proposition \ref{prop 3.1}.
\end{proof}

\begin{proposition}
\label{prop 3.3} Let $f:I\rightarrow \lbrack 0,\infty )$ be a three times
differentiable mapping on $I^{\circ }$ with $a,b\in I^{\circ }$ such that $%
a<b.$ If $\left\vert f^{\prime \prime \prime }\right\vert ^{q}$ is $\log -$%
convex function with $f^{\prime \prime \prime }\in L_{1}([a,b])$ for some
fixed $q\geq 1,$ then for every division $d$ of $[a,b],$ the midpoint error
estimate satisfies%
\begin{equation*}
\left\vert E(f,d)\leq \right\vert \frac{1}{96}\left( \frac{1}{4}\right) ^{1-%
\frac{1}{q}}\sum_{i=0}^{n-1}\left( x_{i+1}-x_{i}\right) ^{4}\left\{
\left\vert f^{\prime \prime \prime }(x_{i+1})\right\vert \left( \mu
_{1,q}\right) ^{\frac{1}{q}}+\left\vert f^{\prime \prime \prime
}(x_{i})\right\vert \left( \mu _{2,q}\right) ^{\frac{1}{q}}\right\} 
\end{equation*}%
where%
\begin{equation*}
\mu _{1,q}=\frac{2K_{1}^{\frac{q}{2}}\left( q\ln K_{1}-6\right) }{\left(
q\ln K_{1}\right) ^{2}}+\frac{48K_{1}^{\frac{q}{2}}\left( q\ln
K_{1}-2\right) }{\left( q\ln K_{1}\right) ^{4}}+\frac{96}{\left( q\ln
K_{1}\right) ^{4}},
\end{equation*}%
\begin{equation*}
\mu _{2,q}=\frac{2M_{1}^{\frac{q}{2}}\left( q\ln M_{1}-6\right) }{\left(
q\ln M_{1}\right) ^{2}}+\frac{48M_{1}^{\frac{q}{2}}\left( q\ln
M_{1}-2\right) }{\left( q\ln M_{1}\right) ^{4}}+\frac{96}{\left( q\ln
M_{1}\right) ^{4}}
\end{equation*}%
and $K_{1},M_{1}$ are as defined in Proposition \ref{prop 3.1}.
\end{proposition}

\begin{proof}
The proof can be maintained by using Corollary \ref{1.3.} like Proposition %
\ref{prop 3.1}.
\end{proof}

\end{document}